\documentclass[11 pt,final, fullpage]{article}

\usepackage{amsmath,amssymb,amsthm,amscd,amsfonts}
\usepackage{epsfig,pinlabel}
\usepackage[hmargin=1.25in, vmargin=1in]{geometry}
\usepackage[font=small,format=plain,labelfont=bf,up,textfont=it,up]{caption}
\usepackage{subcaption}
\usepackage{verbatim}
\usepackage{url}
\usepackage{calc}
\usepackage{graphicx}
\usepackage{multicol}
\usepackage{float}
\usepackage{overpic}
\usepackage[normalem]{ulem}

\newcommand{\urlwofont}[1]{\urlstyle{same}\url{#1}}

\newcommand{\nc}{\newcommand}
\nc{\nt}{\newtheorem}
\nc{\dmo}{\DeclareMathOperator}

\theoremstyle{plain}
\newtheorem{theorem}{Theorem}[section]
\newtheorem{maintheorem}{Theorem}
\newtheorem{prop}[theorem]{Proposition}
\newtheorem{lemma}[theorem]{Lemma}

\newtheorem{corollary}[theorem]{Corollary}

\theoremstyle{definition}

\theoremstyle{remark}

\DeclareMathOperator{\Mod}{Mod}

\dmo{\SMod}{SMod}
\dmo{\PMod}{PMod}
\dmo{\SHomeo}{SHomeo}
\dmo{\SI}{\mathcal{SI}}
\dmo{\SSp}{SSp}
\dmo{\PSp}{PSp}

\DeclareMathOperator{\Sp}{Sp}
\DeclareMathOperator{\GL}{GL}

\newcommand\Z{\ensuremath{\mathbb{Z}}}

\nc{\p}[1]{\noindent {\bf #1.}}
\nc{\margin}[1]{\marginpar{\scriptsize #1}}

\nc{\PartialIBases}{\mathfrak{IB}}
\nc{\PartialIBasesEx}{\widehat{\mathfrak{IB}}}
\nc{\PartialBases}{\mathfrak{B}}
\nc{\Building}{\mathfrak{T}}
\nc{\height}{\ensuremath{\text{ht}}}
\nc{\Poset}{\mathfrak{P}}
\nc{\Field}{\mathbb{F}}
\nc{\Link}{\ensuremath{\text{Link}}}
\nc{\Star}{\ensuremath{\text{Star}}}

\nc{\SymTorelli}{\ensuremath{\mathcal{SI}}}
\nc{\BTorelli}{\ensuremath{\mathcal{BI}}}
\dmo{\Braid}{\ensuremath{B}}
\dmo{\PureBraid}{\ensuremath{PB}}
\nc{\Hyper}{\ensuremath{\iota}}

\nc{\BigFreeProd}{\mathop{\mbox{\Huge{$\ast$}}}}

\nc{\Quotient}{\ensuremath{\mathcal{Q}}}
\nc{\QuotientEx}{\ensuremath{\widehat{\mathcal{Q}}}}

\nc{\Presentation}[2]{\ensuremath{\text{$\langle #1$ $|$ $#2 \rangle$}}}
\nc{\SpGen}{\ensuremath{S_{\text{Sp}}}}
\nc{\SpRel}{\ensuremath{R_{\text{Sp}}}}
\nc{\QGen}{\ensuremath{S_{\mathcal{Q}}}}
\nc{\QRel}{\ensuremath{R_{\mathcal{Q}}}}
\nc{\PBs}{\ensuremath{T}}
\nc{\Qs}{\ensuremath{\overline{s}}}

\dmo{\PB}{PB}
\nc{\BIredg}{\mathcal{BI}_{2g+1}^{\text{red}}}
\nc{\BI}{\mathcal{BI}}
\dmo{\D}{D}
\dmo{\Stab}{Stab}
\dmo{\Surger}{Surger}
\nc{\I}{\mathcal{I}}

\nc{\spanmap}{span}

\nc{\genbygen}[2]{\premonoid{#1}{#2}}
\nc{\premonoid}[2]{#1 \circledcirc #2}
\nc{\monoid}[2]{#1 \odot #2}

\nc{\G}{\Gamma}
\nc{\raag}{A_\Gamma}
\nc{\racg}{W_\G}
\nc{\raagdelt}{A_\Delta}
\dmo{\Aut}{Aut}
\dmo{\Out}{Out}
\nc{\Autraag}{\Aut(\raag)}
\nc{\Outraag}{\Out(\raag)}
\nc{\diag}{D_\G}
\nc{\Autraagdelt}{\Aut(A_\Delta)}
\nc{\glk}{\GL(k,\mathbb{Z})}
\nc{\gln}{\GL(n,\mathbb{Z})}
\nc{\glnt}{\GL(n,\mathbb{Z} / 2)}
\nc{\glkdelt}{\GL(k |\Delta| ,\mathbb{Z})}
\nc{\gldelt}{\GL(|\Delta| ,\mathbb{Z})}
\nc{\zkdelt}{\mathbb{Z}^{k|\Delta|}}
\nc{\join}{\mathcal{J}}
\nc{\pc}{\mathrm{PC}}
\dmo{\lk}{lk}
\dmo{\st}{st}
\dmo{\Inn}{Inn}

\nc{\pia}{\Pi \mathrm{A}}
\nc{\piaG}{\pia_\G}
\nc{\ppia}{\mathrm{P \Pi A}}
\nc{\ppiaG}{\mathrm{P \Pi A}_\G}
\nc{\epia}{\mathrm{E \Pi A}}
\nc{\epiaG}{\mathrm{E \Pi A}_\G}
\nc{\ptor}{\mathcal{PI}}
\nc{\ptorG}{\mathcal{PI}_\G}
\nc{\CGi}{C_\G(\iota)}
\nc{\pbc}{\mathfrak{B}^\pi}
\dmo{\Cay}{Cay}
\dmo{\rev}{rev}
\nc{\Autfn}{\Aut(F_n)}
\dmo{\supp}{supp}
\dmo{\rk}{\mathrm{rk}}

\include{diagram}
\title{\vspace{-30pt} Palindromic automorphisms of right-angled Artin groups}

\author{Neil J. Fullarton and Anne Thomas\vspace{-6pt}}

\begin{document}

\newcounter{enumi_saved}

\maketitle

\vspace{-23pt}
\begin{abstract}We introduce the palindromic automorphism group and the palindromic Torelli group of a right-angled Artin group $\raag$. The palindromic automorphism group $\pia_\G$ is related to the principal congruence subgroups of $\gln$ and to the hyperelliptic mapping class group of an oriented surface, and sits inside the centraliser of a certain hyperelliptic involution in $\Autraag$. We obtain finite generating sets for $\pia_\G$ and for this centraliser, and determine precisely when these two groups coincide. We also find generators for the palindromic Torelli group. \end{abstract}

\section{Introduction}\label{Intro}

 Let $\Gamma$ be a finite simplicial graph, with vertex set $V = \{ v_1, \dots, v_n \}$. Let $E \subset V \times V$ be the edge set of $\Gamma$. The graph $\Gamma$ defines the \emph{right-angled Artin group} $\raag$ via the presentation
\[ \raag = \langle v_i \in V \mid [v_i,v_j] = 1 \mbox{ iff } (v_i,v_j) \in E \rangle . \]

One motivation, among many, for studying right-angled Artin groups and their automorphisms (see Agol~\cite{Ago13} and Charney~\cite{Cha07} for others) is that the groups $\raag$ and $\Autraag$ allow us to interpolate between families of groups that are classically well-studied: we may pass between the free group $F_n$ and free abelian group $\Z^n$, between their automorphism groups~$\Autfn$ and $\Aut(\Z^n) = \gln$, and even between the mapping class group~$\Mod(S_g)$ of the oriented surface $S_g$ of genus $g$ and the symplectic group $\Sp(2g, \Z)$ (this last interpolation is explained in \cite{Day09_symp}).  See Section~\ref{Prelims} for background on right-angled Artin groups and their automorphisms.

In this paper, we introduce a new subgroup of $\Autraag$ consisting of so-called `palindromic' automorphisms of $\raag$, which allows us a further interpolation, between certain previously well-studied subgroups of $\Aut(F_n)$ and of $\gln$.  An automorphism $\alpha \in \Autraag$ is said to be \emph{palindromic} if $\alpha(v) \in \raag$ is a palindrome for each $v \in V$; that is, each $\alpha(v)$ may be expressed as a word $u_1 \dots u_k$ on $V^{\pm1}$ such that $u_1 \dots u_k$ and its reverse $u_k \dots u_1$ are identical as words. The collection $\pia_\G$ of palindromic automorphisms is, a priori, only a subset of $\Autraag$.  While it is easy to see that $\pia_\G$ is closed under composition, it is not obvious that it is closed under inversion.  In  Corollary~\ref{piaG} we prove that $\pia_\G$ is in fact a subgroup of $\Autraag$.  We thus refer to $\pia_\G$ as the \emph{palindromic automorphism group of $\raag$}.  

When $\raag$ is free, the group $\pia_\G$ is equal to the palindromic automorphism group $\pia_n$ of~$F_n$, which was introduced by Collins~\cite{Col95}.  Collins proved that $\pia_n$ is finitely presented and provided an explicit finite presentation.  The group $\pia_n$ has also been studied by Glover--Jensen~\cite{GJ00}, who showed, for instance, that it has virtual cohomological dimension~$n - 1$.  At the other extreme, when $\raag$ is free abelian, the group $\pia_\G$ is the principal level 2 congruence subgroup $\Lambda_n[2]$ of $\gln$.  Thus $\pia_\G$ enables us to interpolate between these two classes of groups. 

Let $\iota$ be the automorphism of $\raag$ that inverts each $v \in V$.  In the case that $\raag$ is free, it is easy to verify that the palindromic automorphism group $\pia_\G = \pia_n$ is equal to the centraliser $C_\G(\iota)$  of $\iota$ in $\Autraag$ (hence $\pia_n$ is a group).  For a general $\raag$, we prove that~$\pia_\G$ is a finite index subgroup of $C_\G(\iota)$, by first considering the finite index subgroup of~$\pia_\G$ consisting of `pure' palindromic automorphisms; see Theorem~\ref{exact} and Corollary~\ref{piaG}.  The index of $\pia_\G$ in $C_\G(\iota)$ depends entirely on connectivity properties of the graph $\G$, and we give conditions on $\G$ that are equivalent to the groups $\pia_\G$ and $C_\G(\iota)$ being equal, in Proposition~\ref{adj_dom}.  In particular, there are non-free $\raag$ such that $\pia_\G = C_\G(\iota)$.

The order 2 automorphism $\iota$ is the obvious analogue in $\Autraag$ of the hyperelliptic involution $s$ of an oriented surface $S_g$, since $\iota$ and $s$ act as $-I$ on $H_1(\raag, \Z)$ and $H_1(S_g, \Z)$, respectively. The group $\pia_\G$ also allows us to generalise a comparison made by the first author in \cite[Section 1]{Ful15_torelli} between $\pia_n \leq \Aut (F_n)$ and the centraliser in $\Mod(S_g)$ of the hyperelliptic involution $s$, which demonstrated a deep connection between these groups. Our study of $\pia_\G$ is thus motivated by its appearance in both algebraic and geometric settings.

The main result of this paper finds a finite generating set for $\pia_\G$. Our generating set includes the so-called \emph{diagram automorphisms} of $\raag$, which are induced by graph symmetries of $\G$, and the \emph{inversions} $\iota_j \in \Autraag$, with $\iota_j$ mapping $v_j$ to ${v_j}^{-1}$ and fixing every $v_k \in V \setminus \{v_j \}$. The function $P_{ij}: V \to \raag$ sending $v_i$ to $v_j v_i v_j$ and $v_k$ to $v_k$ ($k \neq i$) induces a well-defined automorphism of $\raag$, also denoted $P_{ij}$, whenever certain connectivity properties of $\Gamma$ hold (see Section~\ref{purepals}). We establish that these three types of palindromic automorphisms suffice to generate $\pia_\G$.

\begin{maintheorem}\label{palgens}The group $\pia_\G$ is generated by the finite set of diagram automorphisms, inversions and well-defined automorphisms $P_{ij}$. \end{maintheorem}

We also obtain a finite generating set for the centraliser $\CGi$, in Corollary~\ref{centgens}, by combining the generating set given by Theorem~\ref{palgens} with a short exact sequence involving $C_\G(\iota)$ and the pure palindromic automorphism group (see Theorem~\ref{exact}).  Our generating set for $\CGi$ consists of the generators of $\pia_\G$, along with all well-defined automorphisms of $\raag$ that map $v_i$ to $v_iv_j$ and fix every $v_k \in V \setminus \{ v_i \}$, for some $i \neq j$ with $[v_i,v_j] = 1$ in $\raag$.   

Further, for any re-indexing of the vertex set $V$ and each $k = 1,\dots, n$, we provide a finite generating set for the subgroup $\pia_\G(k)$ of $\pia_\G$ which fixes the vertices $v_1,\dots,v_k$, as recorded in Theorem~\ref{stabs}. The so-called \emph{partial basis complex} of $\raag$, which is an analogue of the curve complex, has as its vertices (conjugacy classes of) the images of members of~$V$ under automorphisms of $\Autraag$.  This complex has not, to our knowledge, appeared in the literature, but its definition is an easy generalisation of the free group version introduced by Day--Putman \cite{DP13} in order to generate the Torelli subgroup of $\Aut (F_n)$.  A `palindromic' partial basis complex was also used in~\cite{Ful15_torelli} to approach the study of palindromic automorphisms of $F_n$. Theorem~\ref{stabs} is thus a first step towards understanding stabilisers of simplices in the palindromic partial basis complex of $\raag$.

We prove Theorem~\ref{palgens} and our other finite generation results in Section~\ref{Palindromic}, using machinery developed by Laurence~\cite{Lau95} for his proof that $\Autraag$ is finitely generated.  The added constraint for us that our automorphisms be expressed as a product of \emph{palindromic} generators forces a more delicate treatment.  In addition, our proof uses Servatius' Centraliser Theorem~\cite{Ser89}, and a generalisation to $\raag$ of arguments used by Collins~\cite[Proposition 2.2]{Col95} to generate $\pia_n$.
Throughout this paper, we employ a decomposition into block matrices of the image of $\Autraag$ in $\gln$ under the canonical map induced by abelianising $A_\G$; this decomposition was observed by Day \cite{Day09_peak} and by Wade \cite{Wad13}.

We also in this work introduce the \emph{palindromic Torelli group $\ptor_\G$ of $\raag$}, which we define to consist of the palindromic automorphisms of $\raag$ that induce the identity automorphism on $H_1(\raag) = \Z^n$.  The group $\ptor_\G$ is the right-angled Artin group analogue of the hyperelliptic Torelli group $\mathcal{SI}_g$ of an oriented surface $S_g$, which has applications to Burau kernels of braid groups \cite{BMP15} and to the Torelli space quotient of the Teichm\"{u}ller space of $S_g$ \cite{Hai06}. Analogues of these objects exist for right-angled Artin groups (see, for example,~\cite{CSV12}), but are not yet well-developed.  We expect that the palindromic Torelli group will play a role in determining their structure. 

Even in the free group case, where $\ptor_\G$ is denoted by $\ptor_n$, little seems to be known about the palindromic Torelli group.  Collins~\cite{Col95} observed that $\ptor_n$ is non-trivial, and Jensen--McCammond--Meier~\cite[Corollary 6.3]{JMcCM07} proved that $\ptor_n$ is not homologically finite if $n~\geq~3$.  
An infinite generating set for $\ptor_n$ was obtained recently in \cite[Theorem~A]{Ful15_torelli}, and this is made up of so-called \emph{doubled commutator transvections} and \emph{separating $\pi$-twists}.  In Section~\ref{Torelli} we recall and then generalise the definitions of these two classes of free group automorphisms, to give two classes of palindromic automorphisms of a general $\raag$, which we refer to by the same names. As a first step towards understanding the structure of $\ptor_\G$, we obtain an explicit generating set as follows.
\begin{maintheorem}\label{ptorgens}The group $\ptor_\G$ is generated by the set of all well-defined doubled commutator transvections and separating $\pi$-twists in $\pia_\G$. \end{maintheorem}

The generating set we obtain in Theorem~\ref{ptorgens} compares favourably with the generators obtained in~\cite{Ful15_torelli} in the case that $\raag$ is free. Specifically, the generators given by Theorem~\ref{ptorgens} are the images in $\Autraag$ of those generators of $\ptor_n$ that descend to well-defined automorphisms of $\raag$ (viewing $\raag$ as a quotient of the free group $F_n$ on the set~$V$).

The proof of Theorem~\ref{ptorgens} in Section~\ref{Torelli} combines our results from Section~\ref{Palindromic} with results for~$\ptor_n$ from~\cite{Ful15_torelli}.  More precisely, as a key step towards the proof of Theorem~\ref{palgens}, we find a finite generating set for the pure palindromic subgroup of $\pia_\G$ (Theorem~\ref{ppiaGgens}).  We then use these generators to determine a finite presentation for the image $\Theta$ of this subgroup  under the canonical map $\Autraag \to \gln$ (Theorem~\ref{raagcongpres}).  In order to find this finite presentation for $\Theta \leq \gln$, we also need Corollary 1.1 from~\cite{Ful15_torelli}, which leverages the generating set for $\ptor_n$ from~\cite{Ful15_torelli} to obtain a finite presentation for the principal level 2 congruence subgroup $\Lambda_n[2] \leq \gln$.  Finally, using a standard argument, we lift the relators of $\Theta$ to obtain a normal generating set for $\ptor_\G$.

\textbf{Acknowledgements.} The authors would like to thank Tara Brendle and Stuart White, for encouraging their collaboration on this paper.  We also thank an anonymous referee for spotting a gap in the proof of Proposition~\ref{cliqueform} in an earlier version, and for other helpful comments.

\section{Preliminaries}\label{Prelims}

In this section we give definitions and some brief background on right-angled Artin groups and their automorphisms.   Throughout this section and the rest of the paper, we continue to use the notation introduced in Section~\ref{Intro}. We will also frequently use $v_i \in V$ to denote both a vertex of the graph $\Gamma$ and a generator of $\raag$, and when discussing a single generator we may omit the index $i$.  Section~\ref{Graph} recalls definitions related to the graph $\G$ and Section~\ref{combis} recalls some useful combinatorial results about words in the group $\raag$.  In Section~\ref{Autraag} we recall a finite generating set for $\Autraag$ and some important subgroups of $\Autraag$, and in Section~\ref{Blocks} we recall a matrix block decomposition for the image of $\Autraag$ in $\gln$.  

\subsection{Graph-theoretic notions}\label{Graph}

We briefly recall some graph-theoretic definitions, in particular the domination relation on vertices of $\Gamma$.

The \emph{link} of a vertex $v \in V$, denoted $\lk(v)$, consists of all vertices adjacent to $v$, and the \emph{star} of $v \in V$, denoted $\st(v)$, is defined to be $\lk(v) \cup \{v\}$. We define a relation $\leq$ on $V$, with $u \leq v$ if and only if $\lk (u) \subset \st (v)$. In this case, we say $v$ \emph{dominates} $u$, and refer to $\leq$ as the \emph{domination} relation \cite{KMLNR80}, \cite{Lau95}. Figure~\ref{dom} demonstrates the link of one vertex being contained in the star of another. Note that when $u \leq v$, the vertices $u$ and $v$ may be adjacent in $\Gamma$, but need not be. To distinguish these two cases, we will refer to \emph{adjacent} and \emph{non-adjacent} domination.
\begin{figure}[ht]
\begin{center}
\begin{overpic}[width=2in]{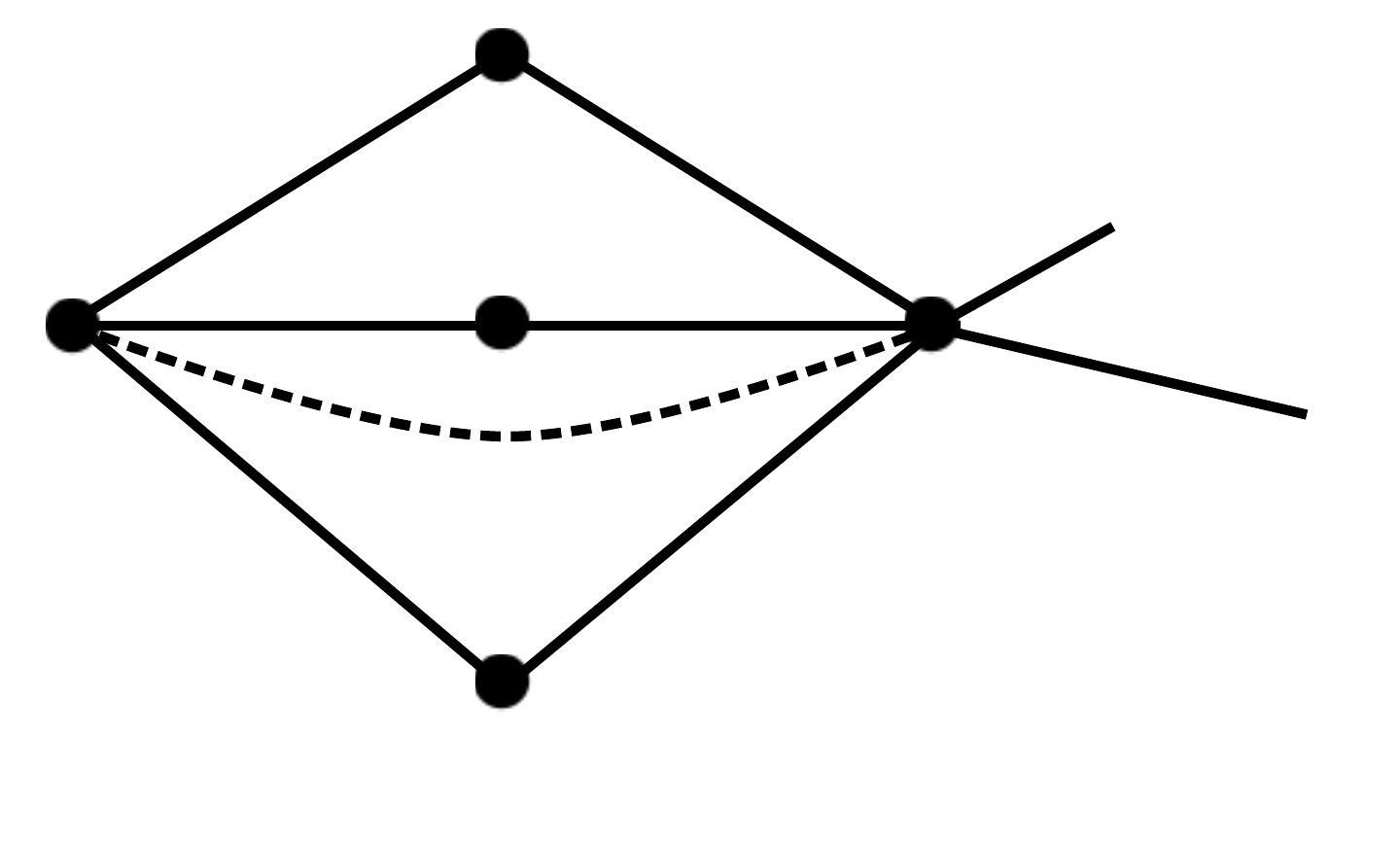}
\put(-3,35){$u$}
\put(64,41){$v$}
\end{overpic}
\caption{An example of a vertex $u$ being dominated by a vertex $v$. The dashed edge is meant to emphasise that $u$ and $v$ may be adjacent, but need not be.}
\label{dom}
\end{center}
\end{figure}

Domination in the graph $\Gamma$ may be used to define an equivalence relation $\sim$ on the vertex set $V$, as follows. We say $v_i \sim v_j$ if and only if $v_i \leq v_j$ and $v_j \leq v_i$, and write $[v_i]$ for the equivalence class of $v_i \in V$ under $\sim$. We also define an equivalence relation $\sim'$ by~$v_i \sim' v_j$ if and only if $[v_i] = [v_j]$ and $v_iv_j = v_jv_i$, writing $[v_i]'$ for the equivalence class of~$v_i \in V$ under~$\sim'$. We refer to $[v_i]$ as the \emph{domination class of $v_i$} and to $[v_i]'$ as the \emph{adjacent domination class of $v_i$}. Note that the vertices in $[v_i]$ necessarily span either an edgeless or a complete subgraph of $\Gamma$; in the former case, we will call $[v_i]$ a \emph{free domination class}, while in the latter, where $[v_i] = [v_i]'$, we will call $[v_i]$ an \emph{abelian domination class}. 

\subsection{Word combinatorics in right-angled Artin groups \label{combis}}

In this section we recall some useful properties of words on $V^{\pm1}$, which give us a measure of control over how we express group elements of $\raag$.  We include the statement of Servatius' Centraliser Theorem~\cite{Ser89} and of a useful proposition of Laurence from~\cite{Lau95}.

First, a word on $V^{\pm1}$ is \emph{reduced} if there is no shorter word representing the same element of~$\raag$. Unless otherwise stated, we shall always use reduced words when representing members of $\raag$.  Now let $w$ and $w'$ be words on $V^{\pm1}$. We say that $w$ and $w'$ are \emph{shuffle-equivalent} if we can obtain one from the other via repeatedly exchanging subwords of the form $uv$ for $vu$ when $u$ and $v$ are adjacent vertices in $\Gamma$. Hermiller--Meier \cite{HM95} proved that two reduced words $w$  and $w'$ are equal in $\raag$ if and only if $w$ and $w'$ are shuffle-equivalent, and also showed that any word can be made reduced by a sequence of these shuffles and cancellations of subwords of the form $u ^{\epsilon} u^{-\epsilon}$ ($u \in V$, $\epsilon \in \{ \pm 1\}$). This allows us to define the \emph{length} of a group element $w \in \raag$ to be the number of letters in a  reduced word representing~$w$, and the \emph{support} of $w \in \raag$, denoted $\supp(w)$, to be the set of vertices~$v \in V$ such that $v$ or $v^{-1}$ appears in a reduced word representing $w$. We say $w \in \raag$ is \emph{cyclically reduced} if it cannot be written in reduced form as $vw'v^{-1}$, for some $v \in V^{\pm1}$, $w' \in \raag$.

Servatius \cite[Section III]{Ser89} analysed centralisers of elements in arbitrary $\raag$, showing that the centraliser of any $w \in \raag$ is again a (well-defined) right-angled Artin group, say $A_\Delta$.  Laurence \cite{Lau95} defined the \emph{rank} of $w \in \raag$ to be the number of vertices in the graph $\Delta$ defining $A_\Delta$. We denote the rank of $w \in \raag$ by $\rk(w)$.

In order to state his theorem on centralisers in $\raag$, Servatius \cite{Ser89} introduced a canonical form for any cyclically reduced $w \in \raag$, which Laurence \cite{Lau95} calls a \emph{basic form} of $w$. For this, partition the support of $w$ into its connected components in $\Gamma^c$, the complement graph of $\Gamma$, writing \[ \supp(w) = V_1 \sqcup \dots \sqcup V_k, \] where each $V_i$ is such a connected component. Then we write \[ w = {w_1}^{r_1} \dots {w_k}^{r_k}, \] where each $r_i \in \Z$ and each $w_i \in \langle V_i \rangle$ is not a proper power in $\raag$ (that is, each $|r_i|$ is maximal). Note that by construction, $[w_i,w_j] = 1$ for $1 \leq i < j \leq k$. Thus the basic form of $w$ is unique up to permuting the order of the $w_i$, and shuffling within each $w_i$. With this terminology in place, we now state Servatius' `Centraliser Theorem' for later use.

\begin{theorem}[Servatius, \cite{Ser89}]\label{servcent} Let $w$ be a cyclically-reduced word on $V^{\pm1}$ representing an element of $\raag$. Writing $w = {w_1}^{r_1} \dots {w_k}^{r_k}$ in basic form, the centraliser of $w$ in $\raag$ is isomorphic to \[ \langle w_1 \rangle \times \dots \times \langle w_k \rangle \times \langle \lk (w) \rangle ,\] where $\lk (w)$ denotes the subset of $V$ of vertices which are adjacent to each vertex in $\supp (w)$. 
\end{theorem}

We will also make frequent use of the following result, due to Laurence~\cite{Lau95}, and so state it now for reference.

\begin{prop}[Proposition 3.5, Laurence \cite{Lau95}]\label{Laurence} Let $w \in \raag$ be cyclically reduced, and write $w = {w_1}^{r_1} \dots {w_k}^{r_k}$ in basic form, with $V_i := \supp(w_i)$. Then:
\begin{enumerate} \item $\rk(v) \geq \rk(w)$ for all $v \in \supp(w)$; and
\item if $\rk(v) = \rk(w)$ for some $v \in V_i$, then: \begin{enumerate} \item $v \leq u$ for all $u \in \supp(w)$; 
\item each $V_j$ is a singleton ($j \neq i$); and 
\item $v$ does not commute with any vertex of $V_i \setminus \{v \}$. \end{enumerate}
\end{enumerate}
\end{prop}

Recall that a \emph{clique} in a graph $\Gamma$ is a complete subgraph.  If $\Delta$ is a clique in $\Gamma$ then $A_\Delta$ is free abelian of rank equal to the number of vertices of $\Delta$, so any word supported on $\Delta$ can be written in only finitely many reduced ways.  The set of cliques in $\Delta$ is partially ordered by inclusion, giving rise to the notion of a maximal clique in a graph $\Gamma$.

\subsection{Automorphisms of right-angled Artin groups}\label{Autraag}

In this section we recall a finite generating set for $\Autraag$.  This generating set was obtained by Laurence \cite{Lau95}, confirming a conjecture of Servatius \cite{Ser89}, who had verified that the set generates $\Autraag$ in certain special cases. 

In the following list, the action of each generator of $\Autraag$ is given on $v \in V$, with the convention that if a vertex is omitted from discussion, it is fixed by the automorphism. There are four types of generators:
\begin{enumerate} 
\item \emph{Diagram automorphisms} $\phi$: each $\phi \in \Aut(\Gamma)$ induces an automorphism of $\raag$, which we also denote by $\phi$, mapping $v \in V$ to $\phi(v)$. 
\item \emph{Inversions} $\iota_j$: for each $v_j \in V$, $\iota_j$ maps $v_j$ to ${v_j}^{-1}$.
\item \emph{Dominated transvections} $\tau_{ij}$: for $v_i,v_j \in V$, whenever $v_i$ is dominated by $v_j$, there is an automorphism $\tau_{ij}$ mapping $v_i$ to $v_iv_j$.  We refer to a (well-defined) dominated transvection $\tau_{ij}$ as an \emph{adjacent transvection} if $[v_i,v_j]=1$; otherwise, we say $\tau_{ij}$ is a \emph{non-adjacent transvection}. 
\item \emph{Partial conjugations} $\gamma_{i,D}$: fix $v_i \in V$, and select a connected component $D$ of $\Gamma \setminus \st(v_i)$ (see Figure~\ref{starcut}). The partial conjugation $\gamma_{v_i,D}$ maps every $d \in D$ to $v_id{v_i}^{-1}$.
\end{enumerate}

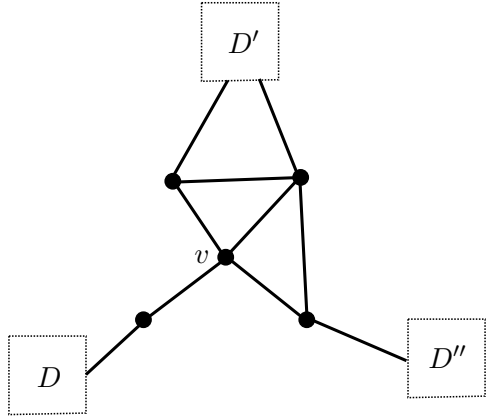
\begin{figure}[ht]
\centering
\begin{overpic}[width=2.5in]{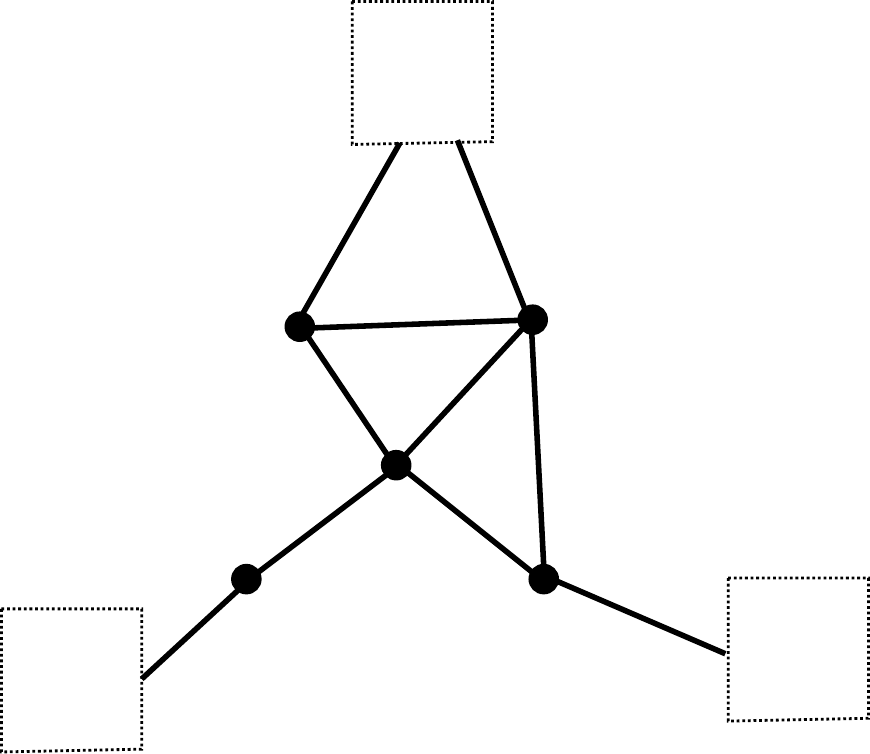}
\put(39,32){$v$}
\put(6,6){$D$}
\put(46,76){$D'$}
\put(88,10){$D''$}
\end{overpic}
\caption{When we remove the star of $v$, we leave three connected components $D$, $D'$ and $D''$.}
\label{starcut}
\end{figure}

We denote by $D_\G$, $I_\G$ and $\pc(\raag)$ the subgroups of $\Autraag$ generated by diagram automorphisms, inversions and partial conjugations, respectively, and by $\Aut^0(\raag)$ the subgroup of $\Autraag$ generated by all inversions, dominated transvections and partial conjugations.

\subsection{A matrix block decomposition}\label{Blocks}

Now we recall a useful decomposition into block matrices of an image of $\Autraag$ inside $\gln$.  This decomposition was observed by Day \cite{Day09_peak} and by Wade \cite{Wad13}.

Let $\Phi : \Autraag \to \gln$ be the canonical homomorphism induced by abelianising $\raag$.  Note that since $D_\Gamma$ normalises $\Aut^0(\raag)$, any $\phi \in \Autraag$ may be written (non-uniquely, in general), as $\phi = \delta \beta$, where $\delta \in D_\G$ and $\beta \in \Aut^0(\raag)$.

 By ordering the vertices of $\G$ appropriately, matrices in $\Phi(\Aut^0(\raag)) \leq \gln$ will have a particularly tractable lower block-triangular decomposition, which we now describe. The domination relation $\leq$ on $V$ descends to a partial order, also denoted $\leq$, on the set of domination classes $V / \sim$, which we (arbitrarily) extend to a total order, \[ [u_1] < \dots < [u_k] \] where $[u_i] \in V / \sim$. This total order may be lifted back up to $V$ by specifying an arbitrary total order on each domination class $[u_i] \in V / \sim$. We reindex the vertices of $\G$ if necessary so that the ordering $v_1, v_2, \dots, v_n$ is this specified total order on $V$. Let $n_i$ denote the size of the domination class $[u_i] \in V / \sim$. Under this ordering, any matrix $M \in \Phi(\Aut^0(\raag))$ has block decomposition:
\[ \begin{pmatrix} M_1 & 0 & 0 & \dots&0
\\ \ast & M_2 & 0 & \dots & 0 \\
\ast & \ast & M_3 & \dots & 0 \\
\vdots & \vdots & \vdots & \ddots & \vdots \\
\ast & \ast & \ast & \dots  & M_k
 \end{pmatrix}, \] where $M_i \in \GL(n_i, \Z)$ and the $(i,j)$ block $\ast$ ($j < i$) may only be non-zero if $u_j$ is dominated by $u_i$ in $\Gamma$. This triangular decomposition becomes apparent when the images of the generators of $\Aut^0(\raag)$ are considered inside $\gln$. The diagonal blocks may be any $M_i \in \GL(n_i, \Z)$, as by definition each domination class gives rise to all $n_i(n_i-1)$ transvections in $\GL(n_i, \Z)$, which, together with the appropriate inversions, generate $\GL(n_i, \Z)$. A diagonal block corresponding to a free domination class will also be called \emph{free}, and a diagonal block corresponding to an abelian domination class will be called \emph{abelian}.
 
This block decomposition descends to an analogous decomposition of the image of $\Aut^0(\raag)$ under the canonical map $\Phi_2$ to $\GL (n, \Z / 2)$, as this map factors through the homomorphism $\gln \to \GL (n, \Z / 2)$ that reduces matrix entries mod 2.

\section{Palindromic automorphisms}\label{Palindromic}

Our main goal in this section is to prove Theorem~\ref{palgens}, which gives a finite generating set for the group of palindromic automorphisms $\pia_\G$.  First of all, in Section~\ref{normalform}, we derive a normal form for group elements $\alpha(v) \in \raag$ where $v \in V$ and $\alpha$ lies in the centraliser~$\CGi$.  In Section~\ref{purepals} we introduce the pure palindromic automorphisms $\ppiaG$, and prove that~$\ppiaG$ is a group by showing that it is a kernel inside $\CGi$.  We then show that $\pia_\G$ is a group, and determine when the groups $\CGi$ and $\pia_\G$ are equal.  The proof of Theorem~\ref{palgens} is carried out in Section~\ref{proofthma}, where the main step is to find a finite generating set for $\ppiaG$.  We also provide finite generating sets for $\CGi$ and for certain stabiliser subgroups of $\pia_\G$.
 
\subsection{The centraliser $\mathbf{\CGi}$ and a clique-palindromic normal form \label{normalform}}

In this section we prove Proposition~\ref{cliqueform}, which provides a normal form for reduced words $w = u_1 \dots u_k$ ($u_i \in V^{\pm1}$) that are equal (in the group $\raag$) to their \emph{reverse}, \[ w^{\mathrm{rev}} := u_k \dots u_1. \]  We then in Corollary~\ref{blocks} derive implications for the diagonal blocks in the matrix decomposition discussed in Section~\ref{Blocks}.  The results of this section will be used in Section~\ref{purepals} below.

Green, in her thesis \cite{Gre90}, established a normal form for elements of $\raag$, by iterating an algorithm that takes a word $w_0$ on $V^{\pm1}$ and rewrites it as $ w_0 = p w_1$ in $\raag$, where $p$ is a word consisting of all the letters of $w_0$ that may be shuffled (as in Section~\ref{combis}) to be the initial letter of $w_0$, and $w_1$ is the word remaining after shuffling each of these letters into the initial segment $p$.  We now use a similar idea for palindromes.  

Let $\iota$ denote the automorphism of $\raag$ that inverts each $v \in V$. We refer to $\iota$ as the \emph{(preferred) hyperelliptic involution of $\raag$}. Denote by $C_\G(\iota)$ the centraliser in $\Autraag$ of $\iota$. Note that this centraliser is far from trivial: it contains all diagram automorphisms, inversions and adjacent transvections in $\Autraag$, and also contains all palindromic automorphisms.  The following proposition gives a normal form for the image of $v \in V$ under the action of some~$\alpha \in \CGi$.

\begin{prop}[Clique-palindromic normal form]\label{cliqueform}Let $\alpha \in C_\Gamma(\iota)$ and $v \in V$. Then we may write \[ \alpha(v) = w_1 \dots w_{k-1} w_k w_{k-1}\dots w_1 , \] where $w_i$ is a word supported on a clique in $\Gamma$ $(1 \leq i \leq k)$, and if $k \geq 3$ then $[w_i, w_{i+1}] \neq 1$ $(1 \leq i \leq k - 2)$. Moreover, this expression for $\alpha(v)$ is unique up to the finitely many rewritings of each word $w_i$ in $\raag$. 
\end{prop}

We refer to this normal form as \emph{clique-palindromic} because the words under consideration, while equal to their reverses in the group $\raag$ as genuine palindromes are, need only be palindromic `up to cliques', as in the expression in the statement of the proposition.

\begin{proof}Suppose $\alpha \in C_\Gamma(\iota)$ and $v \in V$. Write $\alpha(v) = u_1 \dots u_r$ in reduced form, where each $u_i$ is in $V^{\pm1}$. Since $\alpha \iota (v) = \iota \alpha (v)$, we have that \begin{equation}\label{eq:CliqueNormal} u_1 \dots u_r = u_r \dots u_1 \end{equation} in $\raag$. If $\alpha(v)$ is supported on a clique, then there is nothing to show. Otherwise, put $A_1 = \alpha(v)$ and let $Z_1$ be the (possibly empty) subset of $V$ consisting of the vertices in $\supp(A_1)$ which commute with every vertex in $\supp(A_1)$.  We note that $Z_1$ is supported on a clique, and that $Z_1$ is, by assumption, a proper subset of $\supp(A_1)$.

We now rewrite $A_1 = u_1 \dots u_r$ as $w_1 {u_1}' \dots {u_s}'$, where ${u_j}' \in V^{\pm1}$ $(1 \leq j \leq s)$, and $w_1 \in \raag$ is the word consisting of all the $u_i$ which are not in $Z_1^{\pm 1}$ and which may be shuffled to the start of $u_1 \dots u_r$.  That is, $w_1$ consists of all letters $u_i \not \in Z_1^{\pm 1}$ so that if $i \geq 1$, the letter $u_i$ commutes with each of~$u_1,\dots,u_{i-1}$.  Notice that $w_1$ is nonempty since the first $u_i$ which is not in $Z_1$ will be in $w_1$.  By construction, $w_1$ is supported on a clique in $\G$.

Now any $u_i$ that may be shuffled to the start of $u_1 \dots u_r$ may also be shuffled to the end of $u_r \dots u_1$, by~\eqref{eq:CliqueNormal}.   Hence we may also rewrite $A_1$ as $u_1'' \dots u_s'' w_1$ for the same word $w_1$.  Since the support of $w_1$ is disjoint from $Z_1$, the letters of $A_1$ used in the copy of $w_1$ at the start of $w_1 {u_1}' \dots {u_s}'$ are disjoint from the letters of $A_1$ used in the copy of $w_1$ at the end of $u_1'' \dots u_s'' w_1$.  We thus obtain that \[ A_1 = \alpha(v) = w_1 {u_1}'' \dots {u_t}'' w_1 \] in $\raag$, with ${u_i}'' \in V^{\pm1}$. Since  $\alpha \iota (v) = \iota \alpha (v)$, it must be the case that ${u_1}'' \dots {u_t}'' = {u_t}'' \dots {u_1}''$ in $\raag$.

Now put $A_2 = {u_1}'' \dots {u_t}''$, so that $A_1 = w_1 A_2 w_1$.  Note that $\supp(A_2)$ contains $Z_1$.  If $A_2$ is supported on a clique, for example if $\supp(A_2) = Z_1$, then we put $w_2 = A_2$ and are done.  (In this case, $\supp(A_2) = Z_1$ if and only if $w_1$ and $w_2$ commute.)  If $A_2$ is not supported on a clique, we define $Z_2$ to be the vertices in $\supp(A_2)$ which commute with the entire support of $A_2$, and iterate the process described above.  Since each word $w_i$ constructed by this process is nonempty, the word $A_{i+1}$ is shorter than $A_i$, hence the process terminates after finitely many steps.  Notice also that $Z_1 \subseteq Z_2 \subseteq \dots \subseteq Z_i \subseteq \supp(A_{i+1})$, so any letters of $A_i$ which lie in $Z_i$ become part of the word $A_{i+1}$.  In particular, any letter of $A_1 = \alpha(v)$ which is in some $Z_i$, for example a letter in $Z(\raag)$, will end up in the word $w_k$ when the process terminates.

By construction, each $w_i$ is supported on a clique in $\G$.  Now the word $A_{i+1}$ is not supported on a clique if and only if a further iteration is needed, which occurs if and only if $i \leq k - 2$.  In this case, $Z_i$ must be a proper subset of $\supp(A_{i+1})$ and so $w_{i+1}$ does not commute with $w_i$ (the word $w_k$ may or may not commute with $w_{k-1}$).  Thus the expression obtained for $\alpha(v)$ when this process terminates is as in the statement of the proposition.  Moreover, this expression is unique up to rewriting each of the $w_i$, as they were defined in a canonical manner.  This completes the proof.
\end{proof} 

This normal form gives us the following corollary regarding the structure of diagonal blocks in the lower block-triangular decomposition of the image of $\alpha \in \CGi$ under the canonical map $\Phi: \Autraag \to \gln$, discussed in Section~\ref{Blocks}. Recall that $\Lambda_k[2]$ denotes the principal level 2 congruence subgroup of $\glk$.

\begin{corollary}\label{blocks}Write $\alpha \in \CGi$ as $\alpha = \delta \beta$, for some $\beta \in \Aut^0(\raag)$ and $\delta \in D_\Gamma$. Let $M$ be the matrix appearing in a diagonal block of rank $k$ in the lower block-triangular decomposition of $\Phi(\beta) \in \gln$. Then:

 \begin{enumerate} \item if the diagonal block is abelian, then $M$ may be any matrix in $\mathrm{GL}(k, \Z)$; and 
\item if the diagonal block is free then $M$ must lie in $\Lambda_k[2]$, up to permuting columns. \end{enumerate} 
\end{corollary}

\begin{proof}First, note that since $D_\Gamma \leq \CGi$, we must have that $\beta \in \CGi$. We deal with the abelian block case first. The group $\CGi \cap \Aut^0(\raag)$ contains all the adjacent transvections and inversions necessary to generate $\glk$ under $\Phi$, so the matrix $M$ in this diagonal block may be any member of $\glk$.

Now, suppose that the diagonal block is free. Suppose the column of $M$ corresponding to $v \in V$ contains two odd entries, in turn corresponding to vertices $u_1, u_2 \in [v]$, say. This implies that $\beta(v)$ has odd exponent sum of $u_1$ and of $u_2$. Use Proposition~\ref{cliqueform} to write \[ \beta(v) = w_1 \dots w_k \dots w_1 \] in normal form, with each $w_i \in \raag$ being supported on some clique in $\Gamma$. It must be the case that $w_k$ has odd exponent sum of $u_1$ and of $u_2$, since all other $w_i$ ($i \neq k$) appear twice in the normal form expression. Thus $u_1$ and $u_2$ commute. This contradicts the assumption that the diagonal block is free, so there must be precisely one odd entry in each column of~$M$. Hence up to permuting columns, we have $M \in \Lambda_k[2]$.
\end{proof}

\subsection{Pure palindromic automorphisms\label{purepals}}

In this section we introduce the pure palindromic automorphisms $\ppiaG$, which we will see form an important finite index subgroup of $\pia_\G$.  In Theorem~\ref{exact} we prove that $\ppiaG$ is a group, by showing that it is the kernel of the map from the centraliser $\CGi$ to $\GL(n,\Z/2)$ induced by mod 2 abelianisation.  Proposition~\ref{diagram_pure} then says that any element of $\pia_\G$ can be expressed as a product of an element of $\ppiaG$ with a diagram automorphism, and as Corollary~\ref{piaG} we obtain that the collection of palindromic automorphisms $\piaG$ is in fact a group. 
This section concludes by establishing a necessary and sufficient condition on the graph $\Gamma$ for the groups $\pia_\G$ and $\CGi$  to be equal, in Proposition~\ref{adj_dom}.

We define $\ppiaG \subset \piaG$ be the subset of palindromic automorphisms of $\raag$ such that for each $v \in V$, the word $\alpha(v)$ may be expressed as a palindrome whose middle letter is either $v$ or $v^{-1}$. For instance, $I_\G \subset \ppiaG$ but $D_\G \cap \ppiaG$ is trivial.  If $v_i \leq v_j$, there is a well-defined pure palindromic automorphism $P_{ij} := (\iota \tau_{ij})^2$, which sends $v_i$ to $v_j v_i v_j$ and fixes every other vertex in $V$. We refer to $P_{ij}$ as a \emph{dominated elementary palindromic automorphism of $\raag$}.

The following theorem shows that $\ppiaG$ is a group, by establishing that it is a kernel inside~$\CGi$. We will thus refer to $\ppiaG$ as the \emph{pure palindromic automorphism group of~$\raag$}.

\begin{theorem}\label{exact} There is an exact sequence \begin{equation} \label{seq} 1 \longrightarrow \ppiaG \longrightarrow \CGi \longrightarrow \GL(n, \Z / 2). \end{equation} Moreover, the image of $\CGi$ in $\GL(n, \Z / 2)$ is generated by the images of all diagram automorphisms and adjacent dominated transvections in $\Autraag$.\end{theorem}
\begin{proof}Let $\Phi_2 : \Autraag \to \glnt$ be the map induced by the mod 2 abelianisation map $\raag \to (\Z / 2)^n$. We will show that $\ppia_\G$ is the kernel of the restriction of $\Phi_2$ to $\CGi$.

Let $\alpha \in \CGi$.  Note that for each $v \in V$, the element $\alpha(v)$ necessarily has odd length, since~$\alpha(v)$ must survive under the mod 2 abelianisation map $\raag \to (\Z / 2)^n$.   Now for each~$v\in~V$, write $\alpha(v)$ in clique-palindromic normal form $w_1 \dots w_k \dots w_1$, as in Proposition~\ref{cliqueform}.  Both the index $k$ and the word $w_k$ here depend upon $v$, so we write $w(v)$ for the central clique word in the clique-palindromic normal form for $\alpha(v)$.  Then each word~$w(v)$ is a palindrome of odd length which is supported on a clique in $\Gamma$.  It follows that the automorphism $\alpha$ lies in $\ppiaG$ if and only if for each $v \in V$, the exponent sum of $v$ in the word $w(v)$ is odd, and every other exponent sum is even.  Thus $\ppiaG$ is precisely the kernel of the restriction of~$\Phi_2$. 

We now derive the generating set for $\Phi_2(\CGi)$ in the statement of the theorem. Given $\alpha \in \CGi$, write $\alpha = \delta \beta$, where $\delta \in D_\G$ and $\beta \in \Aut^0(\raag)$. We map $\beta$ into $\GL(n, \Z / 2)$ using the canonical map $\Phi_2$, and give $\Phi_2(\beta)$ the lower block-triangular decomposition discussed in Section~\ref{Blocks}.

By Corollary~\ref{blocks}, we can reduce each diagonal block of $\Phi_2(\beta)$ to an identity matrix by composing $\Phi_2(\beta)$ with appropriate members of $\Phi_2(\CGi)$: permutation matrices (in the case of a free block), or images of adjacent transvections (in the case of an abelian block). The resulting matrix $N \in \Phi_2(\CGi)$ lifts to some $\alpha ' \in \CGi$. 

If $N$ has an off-diagonal 1 in its $i$th column, this corresponds to $\alpha'(v_i)$ having odd exponent sum of both $v_i$ and $v_j$, say. Writing $\alpha'(v_i)$ in clique-palindromic normal form $w_1 \dots w_k \dots w_1$, we must have that $v_i$ and $v_j$ both have odd exponent sum in $w_k$, and hence commute, by Proposition~\ref{cliqueform}. The presence of the 1 in the $(j,i)$ entry of $N$ implies that $v_i \leq v_j$, and so we can use the image of the (adjacent) transvection $\tau_{ij}$ to clear it.

Thus we conclude that $\Phi_2(\beta)$ may be written as a product of images of diagram automorphisms and adjacent transvections. Hence $\Phi_2(\CGi)$ is also generated by these automorphisms.
\end{proof}

We now use Theorem~\ref{exact} to prove that the collection of palindromic automorphisms $\piaG$ is a subgroup of $\Autraag$. We will require the following result.

\begin{prop}\label{diagram_pure} Let $\alpha \in \Autraag$ be palindromic.  Then $\alpha$ can be expressed as $\alpha = \delta \gamma$ where $\gamma \in \ppiaG$ and $\delta \in \diag$.
\end{prop}
\begin{proof} Let $\alpha \in \piaG$. Define a function $\delta:V \to V$ by letting $\delta(v)$ be the middle letter of a reduced palindromic word representing $\alpha(v)$. Note that $\delta$ is well-defined, because all reduced expressions for $\alpha(v)$ are shuffle-equivalent, and in any such reduced expression there is exactly one letter with odd exponent sum. The map $\delta$ must be bijective, otherwise the image of $\alpha$ in $\GL(n, \Z / 2)$ would have two identical columns. We now show that $\delta$ induces a diagram automorphism of $\raag$, which by abuse of notation we also denote $\delta$. 

Since $\delta : V \to V$ is a bijection and $\Gamma$ is simplicial, it suffices to show that $\delta$ induces a graph endomorphism of $\Gamma$. Suppose that $u, v \in V$ are joined by an edge in $\Gamma$. Then $[\alpha(v),\alpha(u)]=1$, and so we apply Servatius' Centraliser Theorem (Theorem~\ref{servcent}). Write $\alpha(u)$ in basic form ${w_1}^{r_1} \dots {w_s}^{r_s}$ (see Section~\ref{combis}). Since $\alpha(u)$ is a palindrome, all but one of these $w_i$ will be an even length palindrome, and exactly one will be an odd length palindrome, with odd exponent sum of $\delta(u)$. We know by the Centraliser Theorem that $\alpha(v)$ lies in
\[ \langle w_1 \rangle \times \dots \times \langle w_s \rangle \times \langle \lk(\alpha(u)) \rangle .\] Since $\delta(v) \neq \delta(u)$, the only way $\alpha(v)$ can have an odd exponent of $\delta(v)$ is if $\delta(v) \in \lk(\alpha(u))$. In particular, $[\delta(v),\delta(u)] = 1$. Thus $\delta$ preserves adjacency in $\Gamma$ and hence induces a diagram automorphism.

The proposition now follows, setting $\gamma =  \delta^{-1} \alpha \in \ppiaG$.
\end{proof}

The following corollary is immediate.

\begin{corollary} \label{piaG}  The set $\piaG$ forms a group.  Moreover, this group splits as $\ppiaG \rtimes \diag$. \end{corollary}

We are now able to determine precisely when the groups $\piaG$ and $\CGi$ appearing in the exact sequence \eqref{seq} in the statement of Theorem~\ref{exact} are equal.

\begin{prop}\label{adj_dom} The groups $\piaG$ and $\CGi$ are equal if and only if $\Gamma$ has no adjacent domination classes. \end{prop}

\begin{proof}  If $\G$ has an adjacent domination class, then the adjacent transvections to which it gives rise are in $\CGi$ but not in $\piaG$.

For the converse, suppose $\alpha \in \CGi \setminus \piaG$. Write $\alpha = \delta \beta$, where $\delta \in \diag$ and $\beta \in \Aut^0(\raag)$, as in the proof of Theorem~\ref{exact}. Note that since $\diag \leq \CGi$ we have that $\beta \in \CGi$.  There must be a $v \in V$ such that $\beta(v)$ has at least two letters of odd exponent sum, say $u_1$ and~$u_2$, as otherwise $\alpha$ would lie in $\piaG$. Recall that $u_1$ and $u_2$ must commute, as they both must appear in the central clique word of the clique-palindromic normal form of $\beta(v)$, in order to have odd exponent.

Consider $\Phi(\beta)$ in $\gln$ under our usual lower block-triangular matrix decomposition, discussed in Section~\ref{Blocks}. It must be the case that both $u_1$ and $u_2$ dominate $v$. This is because the odd entries in the column of $\Phi(\beta)$ corresponding to $v$ that arise due to $u_1$ and~$u_2$ either lie in the diagonal block containing $v$, or below this block. In the former case, this gives $u_1, u_2 \in [v]$, while in the latter, the presence of non-zero entries below the diagonal block of $v$ forces $u_1, u_2 \geq v$ (as discussed in Section~\ref{Blocks}). If $v$ dominates $u_1$, say, in return, then we obtain $u_1 \leq v \leq u_2$, and so by transitivity $u_1$ is (adjacently) dominated by $u_2$, proving the proposition in this case. 

Now consider the case that neither $u_1$ nor $u_2$ is dominated by $v$. By Corollary~\ref{blocks}, we may carry out some sequence of row operations to $\Phi(\beta)$ corresponding to the images of inversions, adjacent transvections, or $P_{ij}$ in $\Phi(\CGi)$, to reduce the diagonal block corresponding to~$[v]$ to the identity matrix. The resulting matrix lifts to some $\beta' \in \CGi$, such that $\beta'(v)$ has exponent sum 1 of $v$, and odd exponent sums of $u_1$ and of $u_2$. As we argued in the proof of Corollary~\ref{blocks}, this means $u_1, u_2$ and $v$ pairwise commute, and so $v$ is adjacently dominated by $u_1$ (and $u_2$). This completes the proof.
\end{proof}

\subsection{\label{proofthma}Finite generating sets}

In this section we prove Theorem~\ref{palgens} of the introduction, which gives a finite generating set for the palindromic automorphism group $\pia_\G$.  The main step is Theorem~\ref{ppiaGgens}, where we determine a finite set of generators for the pure palindromic automorphism group $\ppiaG$.   We also obtain finite generating sets for the centraliser $\CGi$ in Corollary~\ref{centgens}, and for certain stabiliser subgroups of $\pia_\G$ in Theorem~\ref{stabs}.

\begin{theorem}\label{ppiaGgens}The group $\ppiaG$ is generated by the finite set comprising the inversions and the dominated elementary palindromic automorphisms. \end{theorem}

Before proving Theorem~\ref{ppiaGgens}, we state a corollary obtained by combining Theorems \ref{exact} and~\ref{ppiaGgens}.
\begin{corollary}\label{centgens}The group $\CGi$ is generated by diagram automorphisms, adjacent dominated transvections and the generators of $\ppiaG$. \end{corollary}

Our proof of Theorem~\ref{ppiaGgens} is an adaptation of Laurence's proof \cite{Lau95} of finite generation of~$\Autraag$. First, in Lemma~\ref{simple} below, we show that any $\alpha \in \ppiaG$ may be precomposed with suitable products of our proposed generators to yield what we refer to as a `simple' automorphism of $\raag$ (defined below).  The simple palindromic automorphisms may then be understood by considering subgroups of $\ppiaG$ that fix certain free product subgroups inside~$\raag$; we define and obtain generating sets for these subgroups in Lemma~\ref{palstab}. Combining these results, we complete our proof of Theorem~\ref{ppiaGgens}.

For each $v \in V$, we define $\alpha \in \ppiaG$ to be \emph{$v$-simple} if  $\supp(\alpha(v))$ is connected in $\Gamma^c$.  We say that $\alpha \in \ppiaG$ is \emph{simple} if $\alpha$ is $v$-simple for all $v \in V$. Laurence's definition of a $v$-simple automorphism $\phi \in \Autraag$ is more general and differs from ours, however the two definitions are equivalent when $\phi \in \ppiaG$. 

Let $S$ denote the set of inversions and dominated elementary palindromic automorphisms in $\piaG$ (that is, the generating set for $\ppiaG$ proposed by Theorem~\ref{ppiaGgens}). We say that $\alpha, \beta \in \ppiaG$ are \emph{$\pi$-equivalent} if there exists $\theta \in \langle S \rangle$ such that $\alpha = \beta \theta$.  In other words, $\alpha, \beta \in \ppiaG$ are $\pi$-equivalent if $\beta^{-1}\alpha \in \langle S \rangle$. 

\begin{lemma}\label{simple} Every $\alpha \in \ppiaG$ is $\pi$-equivalent to some simple automorphism  $\chi\in\ppiaG$. \end{lemma}
\begin{proof}  Suppose $\alpha \in \ppiaG$.  We note once and for all that the palindromic word $\alpha(u)$ is cyclically reduced, for any $u \in V$. 

Select a vertex $v \in V$ of maximal rank for which $\alpha(v)$ is not $v$-simple. Now write \[ \alpha(v) = {w_1}^{r_1} \dots {w_s}^{r_s} \] in basic form, reindexing if necessary so that $v \in \supp(w_1)$. The ranks of $v$ and $\alpha(v)$ are equal, since $\alpha$ induces an isomorphism from the centraliser in $\raag$ of $v$ to that of $\alpha(v)$. Hence by Proposition~\ref{Laurence}, parts 2(b) and 2(a) respectively, each $w_i \in \raag$ (for $i > 1$) is some vertex generator in $V$, and $w_i \geq' v$. Moreover, for $i > 1$, each $r_i$ is even, since $\alpha(v)$ is palindromic.

Now, for $i >1$, suppose $w_i \geq ' v$ but $[v]' \neq [w_i]'$. By Servatius' Centraliser Theorem (Theorem~\ref{servcent}), we know that the centraliser of a vertex is generated by its star, and hence conclude that $\rk(w_i) > \rk(v)$. This gives that $\alpha$ is $w_i$-simple, by our assumption on the maximality of the rank of $v$. In basic form, then, \[\alpha(w_i) = p^\ell, \] where $\ell \in \Z$, $p \in \raag$, and $\supp(p)$ is connected in $\Gamma^c$. Note also that $\supp (p)$ contains $w_i$, since $\alpha \in \ppiaG$. 

Suppose there exists $t \in \supp(p) \setminus \{w_i \}$. As for $v$ before, by Proposition~\ref{Laurence}, we have $t \geq w_i$, since $\rk(\alpha(w_i)) = \rk(w_i)$. We know $w_i \geq' v$, and so $t \geq v$. Since $w_i$, $v$ and $t$ are pairwise distinct, this forces $w_i$ and $t$ to be adjacent, which contradicts Proposition~\ref{Laurence}, part 2(c). So \[\alpha(w_i) = {w_i}^{\ell} ,\] and necessarily $\ell = \pm1$. Knowing this, we replace $\alpha$ with $\alpha \beta_i$ where $\beta_i \in \langle S \rangle$ is the palindromic automorphism of the form \[v \mapsto {w_i}^{ \frac{\ell r_i}{2}} v {w_i}^{\frac{\ell r_i}{2}}.\] By doing this for each such $w_i$, we ensure that any $w_i$ that strictly dominates $v$ is not in the support of $\alpha \beta_i (v)$. Note $\alpha(v') = \alpha \beta_i (v')$ for all $v' \neq v$.

If $s=1$, then $\alpha$ is $v$-simple, so by our assumption on $v$, we must have $s >1$. Because we have reduced to the case where $w_i \in [v]'$ for $i > 1$, we must have $w_1 = v^{\pm1}$, otherwise we get a similar adjacency contradiction as in the previous paragraph: if there exists $t \in \supp(w_1) \setminus  \{ v \}$, then, as before, $t \geq v$, and since $[w_i]' = [v]'$, this would force $t$ and $v$ to be adjacent. Thus $\alpha(v) \in \langle [v]' \rangle$. Indeed, the discussion in the previous two paragraphs goes through for any $u \in [v]'$, so we may assume that $\alpha(u) \in \langle [v]' \rangle$ for any $u \in [v]'$. Thus $\alpha \langle [v]' \rangle \leq \langle [v]' \rangle$, with equality holding by \cite[Proposition 6.1]{Lau95}.

The group $\langle [v]' \rangle$ is free abelian, and by considering exponent sums, we see that the restriction of $\alpha$ to the group $\langle [v]' \rangle$ is a member of the level 2 congruence subgroup $\Lambda_k[2]$, where $k = | [v]'|$. We know that Theorem~\ref{ppiaGgens} holds in the special case of these congruence groups (see \cite[Lemma 2.4]{Ful15_torelli}, for example), so we can precompose $\alpha$ with the appropriate automorphisms in the set $S$ so that the new automorphism obtained, $\alpha '$, is the identity on $\langle [v]' \rangle$, and acts the same as $\alpha$ on all other vertices in $V$. The automorphisms $\alpha$ and $\alpha '$ are $\pi$-equivalent, and~$\alpha '$ is $v$-simple (indeed: $\alpha ' (v) = v$).

From here, we iterate this procedure, selecting a vertex $u \in V \setminus \{ v \}$ of maximal rank for which $\alpha '$ is not $u$-simple, and so on, until we have exhausted the vertices of $\Gamma$ preventing $\alpha$ from being simple.
\end{proof}

Now, for each $v \in V$, define $\Gamma^v$ be the set of vertices that dominate $v$ but are not adjacent to $v$. Further define $X_v := \{v = v_1, \dots, v_r \} \subseteq \Gamma^v$ to be the vertices of $\Gamma^v$ that are also dominated by $v$. Partition $\Gamma^v$ into its connected components in the graph $\Gamma \setminus \lk(v)$. This partition is of the form
\[ \left ( \bigsqcup_{i=1}^{t} \Gamma_i \right ) \sqcup \left (\bigsqcup_{i=1}^r \{ v_i \} \right ), \] where $ \bigsqcup_{i=1}^{t} \Gamma_i = \Gamma^v \setminus X_v$. Letting $H_i = \langle \Gamma_i \rangle$, we see that
\begin{equation}\label{H} H := \langle \Gamma^v \rangle = H_1 \ast \dots \ast H_t \ast \langle X_v \rangle, \end{equation} where $F_r := \langle X_v \rangle$ is a free group of rank $r$.  Notice that $H$ is itself a right-angled Artin group.

The final step in proving Theorem~\ref{ppiaGgens} requires a generating set for a certain subgroup of palindromic automorphisms in $\Aut(H)$, which we now define. Let $\mathcal{Y}$ denote the subgroup of $\Aut (H)$ consisting of the pure palindromic automorphisms of $H$ that restrict to the identity on each $H_i$. The following lemma says that this group is generated by its intersection with the finite list of generators stated in Theorem~\ref{ppiaGgens}.  In the special case when there are no $H_i$ factors in the free product~\eqref{H} above, this result was established by Collins~\cite{Col95}.  Our proof is a generalisation of his.

\begin{lemma}\label{palstab}The group $\mathcal{Y}$ is generated by the inversions of the free group $F_r$ and the elementary palindromic automorphisms of the form $P(s,t): s \mapsto tst$, where $t \in \Gamma^v$ and $s \in X_v$.
\end{lemma}
\begin{proof} For $\alpha \in \mathcal{Y}$, we define its \emph{length} $l(\alpha)$ to be the sum of the lengths of $\alpha(v_i)$ for each $v_i \in X_v$. We induct on this length. The base case is $l(\alpha) = r$, in which case $\alpha$ is a product of inversions of $F_r$. From now on, assume $l(\alpha) > r$.

Let $L(w)$ denote the length of a word $w$ in the right-angled Artin group $H$, with respect to the vertex set $\Gamma^v$. Suppose for all $\epsilon_i, \epsilon_j \in \{ \pm1 \}$ and distinct $a_i, a_j \in \alpha(\Gamma^v)$ we have
\begin{align} L( a_i^{\epsilon_i} a_j^{\epsilon_j}) > L(a_i) + L(a_j) - 2( \lfloor L(a_i) / 2 \rfloor + 1) , \label{ineq}\end{align} where $\lfloor x \rfloor$ is the integer part of $x \in [0, \infty)$. Conceptually, we are assuming that for every expression $a_i^{\epsilon_i} a_j^{\epsilon_j}$, whatever cancellation occurs between the words $a_i^{\epsilon_i}$ and $a_j^{\epsilon_j}$, more than half of $a_i^{\epsilon_i}$ and more than half of  $a_j^{\epsilon_j}$ survives after all cancellation is complete. 

Fix $v_i \in X_v$ so that $a_i := \alpha(v_i)$ satisfies $L(a_i) > 1$.  Such a vertex $v_i$ must exist, as we are assuming that $l(\alpha) > r$.  Notice that since $L(a_i) > 1$, we have $v_i \neq a_i^{\pm 1}$. Now, any reduced word in $H$ of length $m$ with respect to the generating set $\alpha(\Gamma^v)$ has length at least $m$ with respect to the vertex generators $\Gamma^v$, due to our cancellation assumption. Since $v_i \neq {a_i}^{\pm1 }$, the generator $v_i$ must have length strictly greater than 1 with respect to $\alpha(\Gamma^v)$, and so $v_i$ must have length strictly greater than $1$ with respect to $\G^v$.  But $v_i$ is an element of $\G^v$, which is a contradiction. Therefore, the above inequality \eqref{ineq} fails at least once.

We now argue each case separately. Let $a_i, a_j \in \alpha(\G^v)$ be distinct and write $$a_i = \alpha(v_i) = w_i {v_i}^{\eta_i} {w_i}^{\rev} \quad \mbox{and} \quad a_j = \alpha(v_j) = w_j {v_j}^{\eta_j} {w_j}^{\rev},$$ where  $v_i, v_j \in \Gamma^v$, $w_i$, $w_j \in H$ and $\eta_i,\eta_j \in \{\pm 1\}$.  Suppose the inequality~\eqref{ineq} fails for this pair when $\epsilon_i = \epsilon_j = 1$. Then it must be the case that $w_j = ({w_i}^{\rev})^{-1} v_i^{-\eta_i}z$, for some $z \in H$, since $H$ is a free product. In this case, replacing $\alpha$ with $\alpha P(v_j,v_i) = \alpha P_{ji}$ decreases the length of the automorphism. We reduce the length of $\alpha$ in the remaining cases as follows:
\begin{itemize}
\item For $\epsilon_i = \epsilon_j = -1$, replace $\alpha$ with $\alpha \iota_j {P(v_j,v_i)}^{-1} = \alpha \iota_j P_{ji}^{-1}$.

\item For $\epsilon_i = -1$ and $\epsilon_j = 1$, or vice versa, replace $\alpha$ with $\alpha \iota_j P(v_j,v_i) = \alpha \iota_j P_{ji}$. \end{itemize}

By induction, we have thus established the proposed generating set for the group $\mathcal{Y}$.
\end{proof}

We now prove Theorem~\ref{ppiaGgens}, obtaining a finite generating set for the group $\ppiaG$.
\begin{proof}[Proof of Theorem \ref{ppiaGgens}] Let $S$ denote the set of inversions and dominated elementary palindromic automorphisms in $\ppiaG$. By Lemma~\ref{simple}, all we need do is write any simple $\alpha \in \ppiaG$ as a product of members of $S^{\pm1}$. 

Let $v$ be a vertex of maximal rank that is not fixed by $\alpha$. Define $\Gamma^v$, its partition, and the free product it generates using the same notation as in the discussion before the statement of Lemma~\ref{palstab}. By maximality of the rank of $v$, any vertex of any $\Gamma_i$ must be fixed by $\alpha$ (since it has rank higher than that of $v$). By Lemma 5.5 of Laurence and its corollary \cite{Lau95}, we conclude that (for this $v$ we have chosen), $\alpha(H) = H$. 

This establishes that $\alpha$ restricted to $H$ lies in the group $\mathcal{Y} \leq \Aut(H)$, for which Lemma~\ref{palstab} gives a generating set. Thus we are able to precompose $\alpha$ with the appropriate members of~$S^{\pm1}$ to obtain a new automorphism $\alpha '$ that is the identity on $H$, and which agrees with~$\alpha$ on $\Gamma \setminus \Gamma^v$. In particular, $\alpha '$ fixes $v$. We now iterate this procedure until all vertices of $\Gamma$ are fixed, and have thus proved the theorem.
\end{proof}

With Theorem~\ref{ppiaGgens} established, we are now able to prove our first main result, Theorem~\ref{palgens}, and so obtain our finite generating set for $\pia_\G$.

\begin{proof}[Proof of Theorem~\ref{palgens}] By Corollary~\ref{piaG}, we have that $\pia_\G$ splits as  \[ \pia_\G \cong \ppiaG \rtimes D_\Gamma,\] and so to generate $\pia_\G$, it suffices to combine the generating set for $\ppiaG$ given by Theorem~\ref{ppiaGgens} with the diagram automorphisms of $\raag$. Thus the group $\pia_\G$ is generated by the set of all diagram automorphisms, inversions and well-defined dominated elementary palindromic automorphisms.
\end{proof}

We end this section by remarking that the proof techniques we used in establishing Theorem~\ref{palgens} allow us to obtain finite generating sets for a more general class of palindromic automorphism groups of $\raag$.  Having chosen an indexing $v_1,\dots,v_n$ of the vertex set $V$ of~$\G$, denote by $\pia_\G(k)$ the subgroup of $\pia_\G$ that fixes each of the vertices $v_1, \dots, v_k$. Note that a reindexing of $V$ will, in general, produce non-isomorphic stabiliser groups.  We are able to show that each $\piaG (k)$ is generated by its intersection with the finite set $S$.

\begin{theorem} \label{stabs}The stabiliser subgroup $\piaG (k)$ is generated by the set of diagram automorphisms, inversions and dominated elementary palindromic automorphisms that fix each of $v_1, \dots, v_k$. \end{theorem}

Throughout the proof of Theorem~\ref{ppiaGgens}, each time that we precomposed some $\alpha \in \ppiaG$ by an inversion $\iota_i$, an elementary palindromic automorphism~$P_{ij}$, or its inverse~$P_{ij}^{-1}$, it was because the generator $v_i$ was not fixed by $\alpha$. If $v_j \in V$ was already fixed by~$\alpha$, we had no need to use $\iota_j$ or any of the $P_{jk}^{\pm 1}$ ($j \neq k$) in this way. (That this claim holds in the second-last paragraph of the proof of Lemma~\ref{simple}, where we are working in the group~$\Lambda_k[2]$, follows from \cite[Lemma 3.5]{Ful15_torelli}.)  The same is true when we extend $\ppiaG$ to $\piaG$ using diagram automorphisms, in the proof of Theorem~\ref{palgens}. Thus by following the same method as in our proof of Theorem~\ref{palgens}, we are also able to obtain the more general result, Theorem~\ref{stabs}: our approach had already written $\alpha \in \piaG(k)$ as a product of the generators proposed in the statement of Theorem~\ref{stabs}.

\section{The palindromic Torelli group}\label{Torelli}

Recall that we defined the \emph{palindromic Torelli group} $\ptor_\G$ to consist of the palindromic automorphisms of $\raag$ that act trivially on $H_1 (\raag, \Z)$.  Our main goal in this section is to prove Theorem~\ref{ptorgens}, which gives a generating set for $\ptor_\G$.  For this, in Section~\ref{ImagePure} we obtain a finite presentation for the image in $\gln$ of the pure palindromic automorphism group.  Using the relators from this presentation, we then prove Theorem~\ref{ptorgens} in Section~\ref{torsec}.

\subsection{Presenting the image in $\gln$ of the pure palindromic automorphism group}\label{ImagePure}

In this section we prove Theorem~\ref{raagcongpres}, which establishes a finite presentation for the image of the pure palindromic automorphism group $\ppiaG$ in $\gln$, under the canonical map induced by abelianising $\raag$.  Corollary~\ref{SplitPure} then gives a splitting of $\ppiaG$. 

Recall that $\Lambda_n[2]$ denotes the principal level 2 congruence subgroup of $\gln$.  We start by recalling a finite presentation for $\Lambda_n[2]$ due to the first author.  For $1 \leq i \neq j \leq n$, let $S_{ij} \in \Lambda_n[2]$ be the matrix that has 1s on the diagonal and 2 in the $(i,j)$ position, with 0s elsewhere, and let $Z_i \in \Lambda_n[2]$ differ from the identity matrix only in having $-1$ in the $(i,i)$ position. Theorem \ref{congpres} gives a finite presentation for $\Lambda_n[2]$ in terms of these matrices. 
\begin{theorem}[Fullarton \cite{Ful15_torelli}]\label{congpres}The principal level 2 congruence group $\Lambda_n[2]$ is generated by $$\{ S_{ij}, Z_i \mid 1 \leq i \neq j \leq n \},$$ subject to the defining relators

\begin{multicols}{2}
\begin{enumerate}
    \item ${Z_i}^2$
    \item $[Z_i,Z_j]$
    \item $(Z_iS_{ij})^2$
    \item $(Z_jS_{ij})^2$
    \item $[Z_i, S_{jk}]$
    \item $[S_{ki}, S_{kj}]$
    \item $[S_{ij},S_{kl}]$
    \item $[S_{ji},S_{ki}]$
    \item $[S_{kj},S_{ji}]{S_{ki}}^{-2}$
    \item $(S_{ij}{S_{ik}}^{-1}S_{ki}S_{ji}S_{jk}{S_{kj}}^{-1})^2$
\end{enumerate}
\end{multicols} where $1 \leq i,j,k,l \leq n$ are pairwise distinct.
  \end{theorem}
  
 We will use this presentation of $\Lambda_n[2]$ to obtain a finite presentation of the image of $\ppiaG$ in $\gln$. Observe that $\iota_j \mapsto Z_j$ and $P_{ij} \mapsto S_{ji}$ ($v_i \leq v_j$) under the canonical map $\Phi : \Autraag \to \gln$. Let $R_\Gamma$ be the set of words obtained by taking all the relators in Theorem \ref{congpres} and removing those that include a letter $S_{ji}$ with $v_i \not \leq v_j$.
  
  \begin{theorem}\label{raagcongpres}The image of $\ppiaG$ in $\gln$ is a subgroup of $\Lambda_n[2]$, with finite presentation 
  \[ \left \langle \{ Z_k, S_{ji} : 1 \leq k \leq n, v_i \leq v_j \} \mid R_\Gamma \right \rangle. \] \end{theorem}

\begin{proof}By Theorem \ref{ppiaGgens}, we know that $\ppiaG \leq \Aut^0(\raag)$, and so matrices in $\Theta:=\Phi(\ppiaG) \leq \gln$ may be written in the lower-triangular block decomposition discussed in Section~\ref{Blocks}. Moreover, the matrix in a diagonal block of rank $k$ in some $A \in \Theta$ must lie in $\Lambda_{k}[2]$. 

We now use this block decomposition to obtain the presentation of $\Theta$ in the statement of the theorem. Observe that we have a forgetful map $\mathcal{F}$ defined on $\Theta$, where we forget the first $k:= |[v_1]|$ rows and columns of each matrix. This is a well-defined homomorphism, since the determinant of a lower block-triangular matrix is the product of the determinants of its diagonal blocks. Let $\mathcal{Q}$ denote the image of this forgetful map, and $\mathcal{K}$ its kernel. We have $\mathcal{K} = \Lambda_k[2] \times \Z^t$, where $t$ is the number of dominated transvections that are forgotten under the map $\mathcal{F}$, and the $\Lambda_k[2]$ factor is generated by the images of the inversions and dominated elementary palindromic automorphisms that preserve the subgroup $\langle [v_1 ] \rangle$.

The group $\Theta$ splits as $\mathcal{K} \rtimes \mathcal{Q}$, with the relations corresponding to the semi-direct product action, and those in the obvious presentation of $\mathcal{K}$, all lying in $R_\G$. Now, we may define a similar forgetful map on the matrix group $\mathcal{Q}$, so by induction $\Lambda$ is an iterated semi-direct product, with a complete set of relations given by $R_\G$.
\end{proof}

Using the above presentation, we are able to obtain the following corollary, regarding a splitting of the group $\ppiaG$. Recall that $I_\G$ is the subgroup of $\Autraag$ generated by inversions. We denote by $\epiaG$ the subgroup of $\ppiaG$ generated by all dominated elementary palindromic automorphisms.

\begin{corollary}\label{SplitPure} The group $\ppiaG$ splits as $\epiaG \rtimes I_\G$. \end{corollary}
\begin{proof}The group $\ppiaG$ is generated by $\epiaG$ and $I_\G$ by Theorem~\ref{ppiaGgens}, and $I_\G$ normalises~$\epiaG$. We now establish that $\epiaG \cap I_\G$ is trivial. Suppose $\alpha \in \epiaG \cap I_\G$. By Theorem~\ref{raagcongpres}, the image of $\alpha$ under the canonical map $\Phi : \Autraag \to \gln$ lies in the principal level 2 congruence group $\Lambda_n[2]$. This implies that $\Phi(\alpha)$ is trivial, since $\Lambda_n[2]$ is itself a semi-direct product of groups containing the images of the groups $\epiaG$ and $I_\G$, respectively: this is verified by examining the presentation of $\Lambda_n[2]$ given in Theorem~\ref{congpres}. So the automorphism $\alpha$ must lie in the palindromic Torelli group $\ptorG$, which has trivial intersection with $I_\G$, and hence $\alpha$ is trivial. \end{proof}

\subsection{\label{torsec}A generating set for the palindromic Torelli group}

Using the relators in the presentation given by Theorem~\ref{congpres}, we are now able to obtain an explicit generating set for the palindromic Torelli group $\ptorG$, and so prove Theorem~\ref{ptorgens}. 

Recall that when $\raag$ is a free group, the elementary palindromic automorphism $P_{ij}$ is well-defined for every distinct~$i$ and $j$. The first author defined \emph{doubled commutator transvections} and \emph{separating $\pi$-twists} in $\Aut(F_n)$ ($n \geq 3$) to be conjugates in $\pia_n$ of, respectively, the automorphisms $[P_{12}, P_{13}]$ and $(P_{23}{P_{13}}^{-1}P_{31}P_{32}P_{12}{P_{21}}^{-1})^2$. The latter of these two may seem cumbersome; we refer to \cite[Section 2]{Ful15_torelli} for a simple, geometric interpretation of separating $\pi$-twists.

The definitions of these generators extend easily to the general right-angled Artin groups setting, as follows. Suppose $v_i \in V$ is dominated by $v_j$ and by $v_k$, for distinct $i$, $j$ and $k$. Then \[\chi_1(i,j,k) : = [P_{ij}, P_{ik}] \in \Autraag\] is well-defined, and we define a \emph{doubled commutator transvection} in $\Autraag$ to be a conjugate in $\pia_\G$ of any well-defined $\chi_1(i,j,k)$. Similarly, suppose $[v_i] = [v_j] = [v_k]$ for distinct $i$, $j$ and $k$. Then \[ \chi_2(i,j,k) := (P_{jk}{P_{ik}}^{-1}P_{ki}P_{kj}P_{ij}{P_{ji}}^{-1})^2 \in \Autraag \] is well-defined, and we define a \emph{separating $\pi$-twist} in $\Autraag$ to be a conjugate  in $\pia_\G$ of any well-defined $\chi_2(i,j,k)$.

We now prove Theorem~\ref{ptorgens}, showing that $\ptorG$ is generated by these two types of automorphisms.

\begin{proof}[Proof of Theorem~\ref{ptorgens}]  Recall that $\Theta:=\Phi(\ppiaG) \leq \gln$.  The images in $\Theta$ of our generating set for $\ppiaG$ (Theorem~\ref{ppiaGgens}) form the generators in the presentation for $\Theta$ given in Theorem \ref{raagcongpres}. Thus using a standard argument (see, for example, the proof of \cite[Theorem 2.1]{MKS76}), we are able to take the obvious lifts of the relators of $\Theta$ as a normal generating set of $\ptorG$ in $\ppiaG$, via the short exact sequence
\[ 1 \longrightarrow \ptorG \longrightarrow \ppiaG \longrightarrow \Theta \longrightarrow 1 . \] The only such lifts and their conjugates that are not trivial in $\ppiaG$ are the ones of the form stated in the theorem. \end{proof}

\bibliography{palbib}

\end{document}